\documentclass[a4paper,11pt]{article}

\usepackage{amsmath}
\usepackage{amsthm}
\usepackage{amsfonts}
\usepackage{amssymb}
\usepackage{t-angles}

\usepackage[all]{xy}
\usepackage{amscd}

\newcommand{\id}{\mathrm{id}}

\newcommand{\Hom}{\operatorname{Hom}}
\newcommand{\End}{\operatorname{End}}
\newcommand{\Aut}{\operatorname{Aut}}
\newcommand{\Ker}{\operatorname{Ker}}

\def\rbiprod{{\cdot\kern-.33em\triangleright\!\!\!<}}
\def\lbiprod{{>\!\!\!\triangleleft\kern-.33em\cdot\, }}

\def\lrbiprod{{\ \cdot\kern-.60em\triangleright\kern-.33em\triangleleft\kern-.33em\cdot\, }}

\def\lprod{{>\!\!\!\triangleleft\kern-.33em\ \, }}

\allowdisplaybreaks
\newtheorem{theorem}{Theorem}[section]
\newtheorem{lemma}[theorem]{Lemma}
\newtheorem{proposition}[theorem]{Proposition}

\theoremstyle{definition}
\newtheorem{definition}[theorem]{Definition}
\newtheorem{remark}[theorem]{Remark}

\begin{document}
\title{Non-abelian extensions and  automorphisms of post-Lie algebras}

\author{Lisi Bai, Tao Zhang\thanks{The corresponding author. Tao Zhang, {E-mail:} zhangtao@htu.edu.cn}\\
{\footnotesize School of Mathematics and Statistics, Henan Normal University, Xinxiang 453007, PR China}}

\date{}

 \maketitle

 \setcounter{section}{0}

\begin{abstract}
In this paper, we introduce the concepts of crossed modules of post-Lie algebras and cat$^1$-post-Lie algebras. It is proved that these two concepts are equivalent to each other. Secondly, we construct a non-abelian cohomology for post-Lie algebras to classify their non-abelian extensions.
At last, we investigate the inducibility problem of a pair of automorphisms for post-Lie algebras and construct a Wells exact sequence to solve it.

\end{abstract}


\section{Introduction}
Post-Lie algebras were introduced by Vallette \cite{Vallette} in connection with the homology of partition posets and the study of Koszul operads.
Post-Lie algebras which are non-associative algebras play an important role in different areas of pure and applied mathematics.
Additionally, the study of universal enveloping algebras of post-Lie algebras and the free post-Lie algebra has been examined in detail \cite{Fard}. The significance of post-Lie algebras in numerical integration methods is discussed in \cite{Curry,M11}.
The description of post-Lie algebras in regularity structures for solving stochastic partial differential equations is given recently in \cite{BK,BHZ,JZ}.
Furthermore, the classification of post-Lie algebras over specific Lie algebras is presented in \cite{Novikov, Burde16}.
Recently, the deep relationship between post-Lie algebras, post-groups and skew braces was found in \cite{AEM,BGST}.

The structure of post-Lie algebras generalizes that of L-R algebras and pre-Lie algebras, see \cite{Loday,M11,Gubarev,sl2}.  They consist of a vector space $A$ equipped with a Lie bracket $[\cdot ,\cdot ]$ and a binary operation $\cdot $ satisfying the following axioms
\begin{align}
a\cdot \lbrack b,c]& =[a\cdot b,c]+[b,a\cdot c],
\\
\lbrack a,b]\cdot c& =ass(a,b,c)-ass(b,a,c).
\end{align}
where $ass(a,b,c)=a\cdot (b\cdot c)-(a\cdot b)\cdot c$ is the associator.
If the bracket $[\cdot ,\cdot ]$ is zero, then the conditions reduce to
$ass(a,b,c)-ass(b,a,c)=0,$
which is a pre-Lie algebra (also called left symmetric algebra) structure on $A$, thus the concept of post-Lie algebras is a naturally generalization of  pre-Lie algebras.

Whitehead initially introduced the concept of crossed modules for groups \cite{Whitehead}, which Gerstenhaber later extended to Lie algebras \cite{Gerstenhaber}. Kassel and Loday  used  crossed modules to interpret the third relative Chevalley-Eilenberg cohomology of Lie algebras \cite{Kassel}.
 Lie crossed modules are also equivalent to the concept of strict Lie 2-algebras \cite{Baez}. Numerous studies have explored crossed modules in various algebraic structures \cite{Baues,Casas,Zhang1,Zhang2}.
 Eilenberg and Maclane were the first to consider the theory of non-abelian extension for abstract groups \cite{abstract groups ab}. 
Karl-Hermann Neeb investigated the non-abelian extensions of topological Lie algebras in \cite{Neeb}. 
 Several studies have focused on non-abelian extensions of various types of algebras, including  Lie algebras \cite{Non-abelian-Lie,Zhang3}, 3-Lie algebras \cite{3-Liealg,3-Lie extending} and  Rota-Baxter Lie algebras \cite{RB Lie}.
Closely related is the automorphism problem for extensions of algebraic structures. Wells first introduced this problem in the context of abstract groups \cite{well}, and subsequent research has extended it to a range of other algebraic structures. For instance, Lie algebras \cite{well Lie alg}, Lie-Yamaguti algebras \cite{well Lie-Yamaguti}  and  Lie superalgebras \cite{well Lie super}.

In this paper, we introduce the concept of crossed modules within the context of post-Lie algebras and examine the relationship between crossed modules and cat$^1$-post-Lie algebras.
Furthermore, we explore the non-abelian extension theory for post-Lie algebras.
We also investigate inducibility of a pair of automorphisms in a non-abelian extension of post-Lie algebras. Let $0 \to H \to E\to A \to 0$ be a non-abelian extension of post-Lie algebras, and let $\Aut_H(E)$ denote the group of all post-Lie algebras automorphisms which keep $H$ invariant. Then there is a group homomorphism $\Phi: \Aut_H(E) \to \Aut(H) \times \Aut(A)$. A pair $(\beta, \alpha) \in \Aut(H)\times \Aut(A)$ is called inducible if it lies in the image of $\Phi$. Theorems \ref{thm-inducibility} and Theorems \ref{thm-inducibility-2} establish the necessary and sufficient conditions for a pair of automorphisms $(\beta, \alpha)$ to be inducible. Therefore, we define a set map $\mathcal{W} : \mathrm{Aut}(H ) \times \mathrm{Aut} (A ) \rightarrow {\mathcal H}^2_{nab} (A , H )$ via a non-abelian $2$-cocycle $(\rho,\phi,\psi,\sigma,\omega)$, then $(\beta, \alpha)$ is inducible  if and only if $(\beta, \alpha)$ is in the kernel of Wells map $\mathcal{W}(\beta, \alpha)$.  Finally, we construct a short exact sequence connecting automorphism groups and the second non-abelian cohomology group.

The organization of this paper is as follows. In Section 2, we recall some basic concepts about post-Lie algebras. In Section 3, we introduce the notion of crossed modules of post-Lie algebras and establish their equivalence with cat$^1$-post-Lie algebras. In Section 4, we introduce non-abelian cohomology for post-Lie algebras to classify their non-abelian extensions. In Section 5, we study the inducibility problem of a pair of automorphisms about non-abelian extensions
of post-Lie algebras and present a specific case concerning abelian extensions of post-Lie algebras.

Throughout the rest of this paper, we work over a fixed field of characteristic 0.
The space of linear maps from a vector space  $V$ to $W$ is denoted by $\Hom(V,W)$,
and the space of linear maps from a vector space  $V$ to itself is denoted by $\End(V)$.
The identity map from  a space  $V$ to itself is denoted by $\id_V$ or simply by $\id$.
\section{Preliminaries}
\begin{definition}[\cite{Vallette}]
	A  post-Lie algebra $(A, \cdot, [\cdot,\cdot])$ is a vector space equipped with a Lie algebra $(A, [\cdot,\cdot])$
 and a bilinear operation $a\cdot b$ satisfies the following compatibility conditions:
\begin{equation}\label{def3}
 [a,b]\cdot c=a\cdot (b\cdot c)-(a\cdot b)\cdot c-b\cdot (a\cdot c) +(b\cdot a)\cdot c ,
 \end{equation}
\begin{equation}\label{def4}
 a\cdot [b,c]=[a\cdot b, c]+[b,a\cdot c],
 \end{equation}
for all $a,b,c\in A $.
\end{definition}
To a post-Lie algebra $(A, \cdot, [\cdot,\cdot])$ one can associate  another bracket, defined as
\begin{align*}
 \{a,b\}=[a,b]+a\cdot b-b\cdot a ,
\end{align*}
 such that $(A,\{\cdot,\cdot\})$ becomes a Lie algebra. The condition \eqref{def3} implies the following condition
\begin{align*}
 \{a,b\}\cdot c=a\cdot(b\cdot c)-b\cdot(a\cdot c).
\end{align*}
Therefore,  $A$ is a left module over the Lie algebra $(A,\{\cdot,\cdot\})$.

A homomorphism from a post-Lie algebra $(A, \cdot, [\cdot,\cdot])$ to a post-Lie algebra $(A', \cdot', [\cdot,\cdot]')$ is a linear map $f: A\to A'$ such that for all $a,b\in A $, the following equations are satisfied:
$$ f([a,b])=[f(a),f(b)]', \quad f(a\cdot b)=f(a)\cdot' f(b).$$

\begin{definition}
Let ${A}$ be a post-Lie algebra and $H$ be a vector space. Then $H$ is called a representation of $A$ if there exists a representation $\rho: A \to \End(H)$ of the Lie algebra $(A, [\cdot,\cdot])$ and linear maps $\psi, \phi: A \to \End(H)$ such that the following equations hold:
\begin{eqnarray}
&&\rho_{(a \cdot b)} (x)=\phi_{a}(\rho_{b}(x))-\rho_{b}(\phi_{a}( x)),\\
&& \psi_{ [a, b]}(x)=\rho_{a}(\psi_{b}(x))- \rho_{b}(\psi_{a}(x)),\\
&& \psi_{(a \cdot b)}(x)=\psi_{b}(\psi_{a}(x)) -\psi_{b}(\rho_{a}(x))  - \psi_{b}(\phi_{a} (x)) + \phi_{a}(\psi_{b}(x)),\\
&&  \phi_{[a, b]} (x)=\phi_{a}( \phi_{b}(x)) - \phi_{(a \cdot b )} (x) - \phi_{b}( \phi_{a} (x)) + \phi_{(b \cdot a )} (x),
\end{eqnarray}
for all $x \in {H}$ and $a,b\in A$.
\end{definition}

The most natural example of a  representation is the following one.
Let $A$ be a post-Lie algebra. Define linear maps $\rho,\psi,\phi: A \to \End(A)$ by
$\rho_x(y)=[x,y], \psi_x(y) = \phi_y(x) = x\cdot y$, for all $x, y\in A$. Then $(\rho, \psi, \phi)$ is a representation of  $A$ which is called the
adjoint representation of  $A$ on itself.
Let $H$ be a representation of a post-Lie algebra. The semi-direct product of the post-Lie algebras $A$ and $H$  is defined on the direct sum space $A\oplus H$ with the binary operation and the Lie bracket given by
\begin{align*}
	(a, x) \cdot (b, y) := \big(a \cdot b,\, \,  \phi_{a}(y) +\psi_{b}(x) \big),\\
	[(a, x), (b, y)] := \big([a, b],\, \,  \rho_{a}(y)-\rho_{b}(x)  \big),
\end{align*}
for all $a, b \in A$ and $x, y \in H$.

\begin{definition}
Let ${A}$ and  ${H}$ be post-Lie algebras. An action of ${A}$ on ${H}$ are linear maps  $\rho:A \to \End(H),$ $\psi: A \to \End(H)$ and $\phi: A \to \End(H),$ such that $H$ is a representation  of $A$ and  the following equations hold:
\begin{eqnarray}
&&\rho_a[x,y]=[\rho_a(x), y]+[x,\rho_a(y)],\\
&&\rho_{a}(x\cdot y)=x\cdot\rho_{a}(y)+[y,\psi_{a}(x)],\\
&& \phi_{a} [x, y]=[\phi_{a}(x), y]+[x, \phi_{a}(y)],\\
&&\psi_{a}[x, y]=x \cdot\psi_{a}(y) - \psi_{a}(x \cdot y) - y \cdot\psi_{a}(x) + \psi_{a}(y \cdot x) ,\\
&& \label{def13}\phi_{a}(x \cdot y)=\rho_{a}(x)\cdot y-\psi_{a}(x)\cdot y+ x \cdot\phi_{a}(y) +\phi_{a}(x)\cdot y,
\end{eqnarray}
for all $x,y\in {H}$ and $a\in {A}$.
\end{definition}

Let $A $ and $H$ be post-Lie algebras and $A$ acts on $H$.
Then $A \oplus H$ forms a post-Lie algebra with the binary operation and the Lie bracket:
\begin{align}
(a, x) \cdot (b, y) := \big(a \cdot b,\, \, x \cdot y+ \phi_{a}(y)+\psi_{b}(x) \big),\\
[(a, x), (b, y)] := \big([a, b],\, \, [x, y]+  \rho_{a}(y) -\rho_{b}(x) \big),
\end{align}
 for all $a, b\in A$ and $x, y\in {H}$. We denote it by $A\ltimes_{(\rho,\phi,\psi)} H$ or $A\ltimes H$ for simplicity.
If $A\ltimes A$, we denote $\rho_{a}(b)=[a,b]$, $\phi_{a}(b)=b\cdot a$, and $\psi_{a}(b)=a\cdot b$.

\section{Crossed modules of post-Lie algebras}
In this section, we introduce the concept of crossed modules of post-Lie algebras and establish a one-to-one correspondence with the concept of cat$^1$-post-Lie algebras.
\begin{definition}
	Let $A $ and $H$ be two post-Lie algebras. A crossed module of post-Lie algebras $(A,H,\partial) $ is a homomorphism $\partial:H\rightarrow A$
	together with  an action $(\rho,\phi,\psi)$ of $A $ on $H$ such that the following conditions hold:
	\begin{eqnarray}
		\label{cm1}&&\partial(\rho_a(x))=[a,\partial(x)],~~~\partial(\psi_a(x))=\partial(x)\cdot a,~~~\partial(\phi_a(x))=a\cdot\partial(x),\\
		\label{cm2}&&\rho_{\partial(x)}y=[x,y],~~~~~~~~~~ \psi_{\partial(y)} x=x\cdot y=\phi_{\partial(x)}y,
	\end{eqnarray}
	for all $a\in A$ and $ x,y\in H.$ Condition (\ref{cm1}) is called equivariance, and condition (\ref{cm2}) is known as Peiffer's identity.
\end{definition}
\begin{proposition}
	Let $A $ and $H$ be two post-Lie algebras with an action $(\rho,\phi,\psi)$ of $A $ on $H$. A homomorphism $\partial:H\rightarrow A$ is a crossed module of post-Lie algebras
	if and only if the maps $(\partial,\id_ H) : H \ltimes H \to {A}\ltimes H$ and $(\id_A, \partial) : {A}\ltimes H \to {A}\ltimes{A}$  are homomorphisms of post-Lie algebras.
\end{proposition}

\begin{proof}
By the definition of homomorphisms in post-Lie algebras, we have
	\begin{eqnarray*}
{(\id_A, \partial) [({a}, {x}), ({b},y)} ]={[(\id_A, \partial)({a}, {x}), (\id_A, \partial)({b},y)]},\\
{(\id_A, \partial) \Big(({a}, {x})\cdot ({b},y)\Big)}
={\Big((\id_A, \partial)({a}, {x})\Big)\cdot \Big((\id_A, \partial)({b},y)\Big)},
\end{eqnarray*}
which implies
\begin{eqnarray*}
 ([{a},{b}], \partial\rho_a(y)-\partial\rho_b(x) + \partial[{x}, y])=([{a}, {b}], [a, \partial (y)]+[\partial({x}), b]+[\partial({x}), \partial(y)]),\\
({a}\cdot{b}, \partial(\phi_a(y))+\partial(\psi_b(x)) + \partial({x}\cdot y))
=({a}\cdot {b}, a\cdot \partial (y)+\partial({x})\cdot b+\partial({x})\cdot \partial(y)).
\end{eqnarray*}
	Thus $(\id_A, \partial)$  is a homomorphism of post-Lie algebras  if and only if $\partial(\rho_a(x))=[a,\partial(x)]$, $\partial(\psi_a(x))=\partial(x)\cdot a$ and  $\partial(\phi_a(x))=a\cdot\partial(x)$.
	Similarly, we get
	\begin{eqnarray*}
{(\partial,\id_ H)[(x_1, y_1) , (x_2, y_2)]}={[((\partial,\id_ H)(x_1, y_1), (\partial,\id_ H) (x_2, y_2)]},\\
{(\partial,\id_ H)\Big( (x_1, y_1) \cdot (x_2, y_2)\Big)}={\Big((\partial,\id_ H)(x_1, y_1)\Big)\cdot\Big( (\partial,\id_ H) (x_2, y_2)\Big)}.
\end{eqnarray*}
Then, we have
	\begin{eqnarray*}
(\partial[x_1, x_2], [x_1, y_2] +[y_1, x_2]+ [y_1,  y_2] )=([\partial (x_1) , \partial (x_2)],\rho_{\partial(x_1)} y_2-\rho_{\partial(x_2)}y_1+ [y_1, y_2]),\\
 (\partial(x_1\cdot x_2), x_1\cdot y_2 +y_1\cdot x_2+ y_1\cdot  y_2 )=(\partial (x_1) \cdot \partial (x_2),\psi_{\partial(x_2)}y_1+\phi_{\partial(x_1)} y_2+ y_1\cdot y_2).
\end{eqnarray*}
	Thus, $(\partial, \id_H)$ is a homomorphism of post-Lie algebras if and only if it satisfies Peiffer's identity.
These are precisely the crossed module conditions \eqref{cm1} and \eqref{cm2}, which completes the proof.
\end{proof}

\begin{definition}
A morphism of post-Lie crossed modules from $(A,H,\partial)$ to $(A',H',\partial')$ consists of a pair of post-Lie homomorphisms $f: A \rightarrow A'$ and $g: H \rightarrow H'$ such that for all $a \in A$ and $x \in H$,
	\begin{eqnarray*}
		&&f\partial=\partial'g,\quad 	f(\rho_a(x))=\rho'_{f(a)}g(x),\\		
		&&f(\phi_a(x))=\phi'_{f(a)}g(x),\quad f(\psi_a(x))=\psi'_{f(a)}g(x).
	\end{eqnarray*}
A morphism  from $(A,H,\partial)$ to $(A',H',\partial')$ is called an isomorphism if  $f$ and $g$ are bijective maps, which is denoted by $(A,H,\partial)\cong (A',H',\partial')$.
\end{definition}
Brown and Loday introduced cat$^n$-groups as internal $n$-fold categories in the category of groups \cite{BL,Loday1}, we introduce the analogous concept of cat$^1$-post-Lie algebras.
\begin{definition}
	A cat$^1$-post-Lie algebra $(A_1,A_0,\sigma,\tau)$ consists of a post-Lie algebra $A_1$ together with a subalgebra
	$A_0$ and structural homomorphisms $\sigma,\tau:A_1\rightarrow A_0$  satisfying
	\begin{eqnarray}
		\label{cat1}&&\sigma|_{A_0}=\tau|_{A_0}=id_{A_0},\\
		\notag			&&[\ker(\sigma), \ker(\tau)]=0=[\ker(\tau), \ker(\sigma)],\\	\label{cat2}&&\ker(\sigma)\cdot \ker(\tau)=0=\ker(\tau)\cdot \ker(\sigma),
	\end{eqnarray}
where $\ker(\sigma)$ and $\ker(\tau)$ are the kernel spaces of $\sigma$ and $\tau$ respectively.
\end{definition}
\begin{definition}
	A morphism of cat$^1$-post-Lie algebras from $(A_1,A_0,\sigma,\tau)$ to $(A_1',A_0',\sigma',\tau')$ is a post-Lie homomorphism $h:A_1\rightarrow A_1'$ such that
	\begin{eqnarray*}
		&&h(A_0)\subseteq A_0',~~~~\sigma'h=h|_{A_0}\sigma,~~~~\tau'h=h|_{A_0}\tau.
	\end{eqnarray*}
A morphism  from $(A_1,A_0,\sigma,\tau)$ to $(A_1',A_0',\sigma',\tau')$  is called an isomorphism if  $h$ is a bijective map, which is denoted by $(A_1,A_0,\sigma,\tau)\cong (A_1',A_0',\sigma',\tau')$.
\end{definition}

\begin{theorem}
	There is a one-to-one correspondence between the category of  crossed modules of post-Lie algebras and cat$^1$-post-Lie algebras.
\end{theorem}
	\begin{proof}
		Given a crossed module of post-Lie algebras $(A,H,\partial)$, the corresponding cat$^1$-post-Lie algebra is \( (A \ltimes H, A, \sigma, \tau) \), where $\sigma, \tau:A \ltimes H\rightarrow A$ are given by: for all  $a,b\in A$ and $x,y\in H$
		$$\sigma(a,x) = a,~~~~\tau(a,x) = a+\partial(x) .$$
It is clear that \( \sigma \) and \( \tau \) are post-Lie homomorphisms with \(\sigma|_{A} = \tau|_{A} = \id_A\). we obtain
		$$\ker\sigma=\{(a,x)\in A \ltimes H \mid \sigma(a,x)=a=0\}=\{(0,x)\in A \ltimes H\},$$
		$$\ker\tau=\{(b,y)\in A \ltimes H \mid \tau(b,y)=b+\partial(y)=0\}=\{(-\partial(y),y)\in A \ltimes H\}.$$
Then,
$$[\ker \sigma , \ker\tau ] = [(0, x),(-\partial(y),y)] = (0,\rho_{\partial(y)}x+[x, y])= 0,$$
$$\ker \sigma \cdot \ker\tau  = (0, x)\cdot(-\partial(y),y) = (0,-\psi_{\partial(y)}x+x\cdot y)= 0,$$
$$\ker \tau \cdot \ker\sigma  = (-\partial(y),y)\cdot(0, x) = (0,-\phi_{\partial(y)}x+y\cdot x)= 0,$$
which implies
	$$[\ker(\sigma), \ker(\tau)]=0=[\ker(\tau), \ker(\sigma)],$$
		$$\ker \sigma \cdot \ker\tau =0=\ker \tau \cdot \ker\sigma.$$
These are exactly the crossed module conditions \eqref{cm1} and \eqref{cm2}. Therefore, \( (A \ltimes H, A, \sigma, \tau) \) constitutes a cat$^1$-post-Lie algebra.
		
Suppose a cat$^1$-post-Lie algebra $(A_1, A_0, \sigma, \tau)$ is given. We can construct a mapping $\tau|_{\ker\sigma} : \ker\sigma \rightarrow A_0$ and define an action $(\rho, \psi, \phi)$ of $A_0$ on $\ker\sigma$ that is induced by the Lie bracket and the bilinear operation in $A_1$.
 We will demonstrate that $(A_0, \ker\sigma, \tau|_{\ker\sigma})$ forms a crossed module of  post-Lie algebras.

Let $x, y \in \ker \sigma$. Then $\tau(x) \in A_0$. Using condition (\ref{cat1}), we have $\tau(\tau(x)) = \tau(x)$ and
$$\tau(\rho_{\tau(x)}y)=[\tau(\tau(x)),\tau(y)]=[\tau(x),\tau(y)],$$
$$\tau(\phi_{\tau(x)}y)=\tau(\tau(x))\cdot\tau(y)=\tau(x)\cdot\tau(y),$$
$$\tau(\psi_{\tau(x)}y)=\tau(y)\cdot\tau(\tau(x))=\tau(y)\cdot\tau(x).$$
Thus, it satisfies the condition (\ref{cm1}). Additionally, since
		$$\tau(\tau(x)-x)=\tau(\tau(x))-\tau(x)=\tau(x)-\tau(x)=0,$$
		it follows that $\tau(x)-x \in \ker\tau.$ According to the condition(\ref{cat2}), we have
$$[(\tau(x)-x), y]=\rho_{\tau(x)}y-[x, y]\in [\ker\tau,\ker\sigma]=0,$$
		$$y\cdot(\tau(x)-x) =\psi_{\tau(x)}y-y\cdot x\in\ker\sigma\cdot \ker\tau=0,$$
		$$(\tau(x)-x)\cdot y=\phi_{\tau(x)}y-x\cdot y\in\ker\tau\cdot \ker\sigma=0.$$
	We have established that the condition (\ref{cm2}) holds. Therefore, $(A_0,\ker\sigma,\tau|_{\ker\sigma})$ is a crossed module of post-Lie algebras.
	
 Let $[(A, H, \partial)]$ denote the isomorphism class of crossed modules of post-Lie algebras and $[(A \ltimes H, A, \sigma, \tau)]$ denote the isomorphism class of $\text{cat}^1$-post-Lie algebras. We define maps between these classes:
	\[
	\eta:[(A, H, \partial)] \rightarrow [(A_1, A_0, \sigma, \tau)] \quad \text{and} \quad \theta:[(A_1, A_0, \sigma, \tau)]\rightarrow [(A, H, \partial)],
	\]
and show that $\eta \theta = \id, \theta  \eta = \id$.

Now suppose $(A, H, \partial) \cong (A', H', \partial')$ via isomorphisms $f: A \rightarrow A'$ and $g: H \rightarrow H'$ satisfying	
	\begin{eqnarray*}
		&&f\partial=\partial'g,\quad 	f(\rho_a(x))=\rho'_{f(a)}g(x),\\		
		&&f(\phi_a(x))=\phi'_{f(a)}g(x),\quad f(\psi_a(x))=\psi'_{f(a)}g(x).
	\end{eqnarray*}
	Let $(A \ltimes H, A, \sigma, \tau)$ and $(A' \ltimes H', A', \sigma', \tau')$ be the $\text{cat}^1$-post-Lie algebras constructed from $(A, H, \partial)$ and $(A', H', \partial')$. Define a linear map $h: A \ltimes H \rightarrow A' \ltimes H'$ by $(a, x) \mapsto ( f(a),g(x))$. Since $f: A \to A'$ and $g: H \to H'$ are bijective, the map $h$ is also a bijection. For any  $a, b \in A$, $x,y \in H$, we have
	\begin{align*}
		h[(a,x),(b,y)]
		&= h\big([a, b],\, \, [x, y]+  \rho_{a}(y) -\rho_{b}(x) \big) \\
		&= (f[a, b], g([x, y]+  \rho_{a}(y) -\rho_{b}(x))) \\
		&= ([f(a), f(b)]', ([g(x), g(y)]'+  \rho'_{f(a)}(g(y)) -\rho'_{f(b)}(g(x)))) \\
		&= [h(a, x) ,h(b, y)]',
	\end{align*}
	and similarly
$h((a,x)\cdot(b,y))= h(a, x)\cdot 'h(b, y).$
	 Moreover,  we have $h \sigma = \sigma'  h$ and $h \tau = \tau' h$. Therefore, $(A \ltimes H, A, \sigma, \tau) \cong (A' \ltimes H', A', \sigma', \tau')$, which shows that $\eta$ is well-defined.

	Suppose $(A_1, A_0, \sigma, \tau) \cong (A_1', A_0', \sigma', \tau')$. Then there exists a post-Lie algebra isomorphism $h: A_1 \rightarrow A_1'$ such that $h \sigma = \sigma' h$ and $h \tau = \tau' h$. Let $(A_0, \ker\sigma, \tau|_{\ker\sigma})$ and $(A_0', \ker\sigma', \tau|_{\ker\sigma}')$ be the crossed modules constructed from $(A_1, A_0, \sigma, \tau)$ and $(A_1', A_0', \sigma', \tau')$.  Define linear maps $f:A_0 \rightarrow A_0'$ and $g:\ker\sigma \rightarrow \ker\sigma'$. Clearly, $(A_0, \ker\sigma, \tau|_{\ker\sigma}) \cong (A_0', \ker\sigma', \tau|_{\ker\sigma}')$.  Thus, $\theta$ is well-defined.
	
	Finally, let $$[(A_1, A_0, \sigma, \tau)] \overset{\theta}{\rightarrow} [(A_0, \ker\sigma, \tau|_{\ker\sigma})] \overset{\eta}{\rightarrow} [(A_0 \ltimes \ker\sigma, A_0, \sigma', \tau')],$$ where the maps $\sigma'$ and $\tau'$ are given by $\sigma'(a,x) = a, \tau'(a,x) = a+\tau|_{\ker\sigma}(x) $. Define a linear map $\hat{h}: A_0 \ltimes \ker\sigma \rightarrow A_1$ by $(a, x) \mapsto a + x$. For any $(a, x) \in \operatorname{Ker} \hat{h}$, this map is injective since $\hat{h}(a, x) = 0$ implies $a = -x \in A_0 \cap \ker \sigma = \{0\}$. 	
	For any $e \in A_1$, we have $\sigma(e) \in A_0$, and
$\sigma(e - \sigma(e)) = \sigma(e) - \sigma(e) = 0,$ which implies $e - \sigma(e) \in \ker \sigma$.  Therefore, there exists $(\sigma(e), e - \sigma(e)) \in A_0 \ltimes \ker\sigma$ such that $\hat{h}(\sigma(e), e - \sigma(e)) = e$, hence $\hat{h}$ is surjective. 	
	Since
$\hat{h}(\sigma'(a, x)) = \hat{h}(a)  \sigma(a ,x) = \sigma(\hat{h}(a, x)),$
	and similarly $\hat{h} \tau' = \tau \hat{h}$, implies that $\hat{h}$ is a post-Lie algebra homomorphism. Consequently,   $(A_1, A_0, \sigma, \tau) \cong (A_0 \ltimes \ker\sigma, A_0, \sigma', \tau')$, and hence $\eta \theta = \text{id}$. Furthermore, it follows that $\theta  \eta = \text{id}$, which completes the proof.
\end{proof}
\section{Non-abelian extension of post-Lie algebras}\label{Extensions1}

In this section, we study non-abelian extensions of a post-Lie algebra by another one. We define the second non-abelian cohomology space that classifies equivalence classes of such extensions.

\begin{definition}\label{2-cocycle}
Let $A,H$ be post-Lie algebras, where $\rho : A \rightarrow \mathrm{End}(H)$ is an action of the Lie algebra $(A, [\cdot,\cdot])$ on $H$. An extension datum of ${A}$ by $H$  is  $\Omega({A}, H)=(\rho,\,\psi,\,\phi,\, \sigma, \,\omega)$ consisting of maps
\begin{eqnarray*}
&&\rho\,: A\times H \to H,\quad\psi: H \times A \to H,
\quad\phi: A \times H \to H,\\
&&\sigma: A\times A\to H, \quad\omega: A\times A\to H,
\end{eqnarray*}
such that the following conditions hold
\begin{enumerate}
\item[(A1)]$\phi_{b}(x \cdot z)=\phi_{b}(x)\cdot z+\rho_{b}(x)\cdot z-\psi_{b} (x) \cdot z+x \cdot\phi_{b}(z) $,
\item[(A2)]$\psi_{c} [x, y]=\psi_{c} (y \cdot x)- y \cdot\psi_{c} (x)-\psi_{c}( x \cdot y)+x \cdot\psi_{c} (y)$,
\item[(A3)]$\psi_{(b \cdot c)} (x)=\psi_{c}(\psi_{b}(x))-\psi_{c} (\rho_{b}(x))-\psi_{c}(\phi_{b}(x))
+\phi_{b}(\psi_{c}(x))-x \cdot\omega(b,c)$,
\item[(A4)]$\phi_{[a, b]}(z)+\sigma(a,b)\cdot z =\phi_{(b \cdot a)}(z) -\phi_{b}(\phi_{a}(z))-\phi_{(a \cdot b)}(z)
+\phi_{a}(\phi_{b}(z))+\omega(b,a) \cdot z\\
-\omega(a,b) \cdot z$,
\item[(A5)]$\psi_{c} (\sigma(a,b))+\omega([a, b],c)=\psi_{c} (\omega(b,a))+\omega(b \cdot a,c)-\phi_{b}(\omega(a,c))\\
 -\omega(b,a \cdot c)-\psi_{c} (\omega(a,b))-\omega(a \cdot b,c)+\phi_{a}(\omega(b,c))+\omega(a,b \cdot c)$,
\item[(A6)]$\rho_{c}(x \cdot y)=x \cdot\rho_{c}(y)+[y, \psi_{c} (x)]$,
\item[(A7)]$\phi_{a} [y, z]=[\phi_{a}(y), z] + [y,\phi_{a}(z)]$,
\item[(A8)]$ \psi_{[b, c]} (x)=\rho_{b}( \psi_{c} (x))-\rho_{c}(\psi_{b} (x))-x \cdot\sigma(b,c)$,
\item[(A9)]$\rho_{(a \cdot c)}(y)=\phi_{a}( \rho_{c}(y))-\rho_{c}(\phi_{a}(y))+[y,\omega(a,c)] $,
\item[(A10)]$\phi_{a}(\sigma(b,c)) +\omega(a,[b, c])=\rho_{b}(\omega(a,c))-\rho_{c}(\omega(a,b))+\sigma(a \cdot b,c) +\sigma(b,a \cdot c)$,
\end{enumerate}
 for all $a, b, c\in A$ and $x, y, z\in {H}$.
\end{definition}

\begin{lemma}\label{def extension}
Let $\Omega({A}, H)$ be an extension datum of ${A}$ by $H$. Define on the vector space ${A}\oplus H$ the operations given by
\begin{eqnarray}
\label{product2}(a, x) \cdot (b, y) := \big(a \cdot b,\, \, x \cdot y+ \phi_{a}(y)+ \psi_{b}(x)  +\omega(a,b)\big),\\
\label{product1}[(a, x), (b, y)] := \big([a, b],\, \, [x, y]+  \rho_{a}(y)-\rho_{b}(x)  +\sigma(a,b)\big).
\end{eqnarray}
Then ${A}\oplus H$ is a post-Lie algebra if and only if the extension datum satisfy conditions (A1)--(A10).
In this case, we  denote the post-Lie algebra on ${A}\oplus H$ by $A\ltimes_{\sigma, \omega} H$.
\end{lemma}

\begin{proof}
%
%

 We have to check the condition \eqref{def3}
\begin{eqnarray*}
&&[(a, x),(b, y)]\cdot (c, z)\\
&=&((b, y)\cdot (a, x))\cdot (c, z) - (b, y)\cdot ((a, x)\cdot (c, z))\\
&&-((a, x)\cdot (b, y))\cdot (c, z)+(a, x)\cdot ((b, y)\cdot (c, z)).
\end{eqnarray*}
By direct computations, the left hand side is equal to
\begin{eqnarray*}
&&[(a, x),(b, y)]\cdot (c, z)\\
&=&\big([a, b],\, \, [x, y]+  \rho_{a}(y)-\rho_{b}(x) +\sigma(a,b)\big)\cdot (c, z)\\
&=&\big([a, b] \cdot c,\, \, ([x, y]+  \rho_{a}(y)-\rho_{b}(x)  +\sigma(a,b))\cdot z+ \psi_{c}([x, y]+  \rho_{a}(y)-\rho_{b}(x) +\sigma(a,b)) \\
&& +  \phi_{[a, b]} (z) +\omega([a, b],c)\big)\\
&=&\big([a, b] \cdot c,\, \, [x, y]\cdot z+ \rho_{a}(y)\cdot z- \rho_{b}(x)\cdot z  +\sigma(a,b)\cdot z+\psi_{c}[x, y] + \psi_{c}(\rho_{a}(y))\\
&& - \psi_{c}(\rho_{b}(x)) +\psi_{c}(\sigma(a,b)) + \phi_{[a, b]} (z) +\omega([a, b],c)\big),
\end{eqnarray*}
and the right hand side is equal to
\begin{eqnarray*}
&&((b, y)\cdot (a, x))\cdot (c, z) - (b, y)\cdot ((a, x)\cdot (c, z))\\
&&-((a, x)\cdot (b, y))\cdot (c, z)+(a, x)\cdot ((b, y)\cdot (c, z))\\
&=&\big(b \cdot a,\, \, y \cdot x+ \psi_{a} (y) +  \phi_{b} (x) +\omega(b,a)\big)\cdot (c, z)\\
&&- (b, y)\cdot\big(a \cdot c,\, \, x \cdot z+ \psi_{c} (x) +  \phi_{a} (z) +\omega(a,c)\big)\\
&&-\big(a \cdot b,\, \, x \cdot y+ \psi_{b} (x) +\phi_{a}( y) +\omega(a,b)\big)\cdot (c, z)\\
&&+(a, x)\cdot \big(b \cdot c,\,  y \cdot z+ \psi_{c} (y) + \phi_{b} (z) +\omega(b,c)\big)\\
&=&\big((b \cdot a) \cdot c,\,  (y \cdot x+ \psi_{a} (y) + \phi_{b} (x) +\omega(b,a)) \cdot z+ \psi_{c} (y \cdot x+ \psi_{a} (y)+ \phi_{b} (x) \\
&&+\omega(b,a))+  \phi_{(b \cdot a)} (z) +\omega(b \cdot a,c)\big)-\big(b \cdot (a \cdot c),\, y \cdot ((x \cdot z) + \psi_{c} (x)+ \phi_{a} (z)\\
&&+\omega(a,c)) + \psi_{(a \cdot c)} (y)+ \phi_{b} (x \cdot z+ \psi_{c} (x)+ \phi_{a} (z)+\omega(a,c)) +\omega(b,a \cdot c)\big)  \\
&&-\big((a \cdot b) \cdot c,\,( x \cdot y+ \psi_{b} (x) +  \phi_{a} (y)+\omega(a,b)) \cdot z + \psi_{c} ( x \cdot y+ \psi_{b} (x) + \phi_{a} (y) \\
&&+\omega(a,b))+ \phi_{(a \cdot b)} (z) +\omega(a \cdot b,c)\big)+\big(a \cdot (b \cdot c),\, \, x \cdot (y \cdot z+ \psi_{c} (y)+ \phi_{b} (z)  \\
&&+\omega(b,c))+ \psi_{(b \cdot c)}(x) + \phi_{a} (y \cdot z+ \psi_{c} (y)+ \phi_{b} (z) +\omega(b,c))+\omega(a,b \cdot c)\big).
\end{eqnarray*}
Thus the two sides are equal to each other if and only if (A1)--(A5) hold. We have to check the condition \eqref{def4}
$$(a, x)\cdot [(b, y),(c, z)]=[(a, x)\cdot (b, y), (c, z)]+[(b, y),(a, x)\cdot (c, z)].$$
By direct computations, the left hand side is equal to
\begin{eqnarray*}
&&(a, x)\cdot [(b, y),(c, z)]\\
&=&(a, x)\cdot \big([b, c],\, \, [y, z]+  \rho_{b}(z)-\rho_{c}(y)   +\sigma(b,c)\big)\\
&=&\big(a \cdot [b, c],\, \, x \cdot ([y, z]+  \rho_{b}(z)-\rho_{c}(y) +\sigma(b,c))+ \psi_{[b, c]}(x)  \\
&&+ \phi_{a} ([y, z]+  \rho_{b}(z)-\rho_{c}(y) +\sigma(b,c))+\omega(a,[b, c])\big),
\end{eqnarray*}
and the right hand side is equal to
\begin{eqnarray*}
&&[(a, x)\cdot (b, y), (c, z)]+[(b, y),(a, x)\cdot (c, z)]\\
&=&[\big(a \cdot b,\, \, x \cdot y+ \psi_{b} (x) + \phi_{a} (y) +\omega(a,b)\big), (c, z)]\\
&&+[(b, y),\big(a \cdot c,\, \, x \cdot z+ \psi_{c}(x) + \phi_{a} (z) +\omega(a,c)\big)]\\
&=&\big([a \cdot b, c],\, \, [x \cdot y+ \psi_{b} (x) + \phi_{a} (y) +\omega(a,b), z]+ \rho_{ (a \cdot b)}(z)- \rho_{c}(x \cdot y+ \psi_{b} (x)   \\
&&+ \phi_{a} (y)+\omega(a,b)) +\sigma(a \cdot b,c)\big)+\big([b, a \cdot c],\, \, [y, x \cdot z+ \psi_{c} (x) + \phi_{a} (z)+\omega(a,c)]  \\
&&+  \rho _{b} (x \cdot z+ \psi_{c} (x) + \phi_{a} (z) +\omega(a,c))-\rho_{ (a \cdot c)}(y)+\sigma(b,a \cdot c)\big).
\end{eqnarray*}
Thus the two sides are equal to each other if and only if (A6)--(A10) hold.
\end{proof}

\begin{definition}\label{equivalent}
	Let  $A $ and $H $ be two post-Lie algebras.
	
	\medskip
	
	$\bullet$ A {\bf non-abelian $2$-cocycle} of $A$ with values in $H $ is $(\rho,\phi,\psi,\sigma,\omega)$ of linear maps such that the conditions (A1)--(A10) are satisfied.
	
	\medskip
	
	$\bullet$ Let $(
	\rho,\phi,\psi,\sigma,\omega)$ and $(
	\rho',\phi',\psi',\sigma',\omega')$ be two non-abelian $2$-cocycles of the post-Lie algebra $A $ with values in $H $.
	They are said to be \textbf{equivalent} if there exists a linear map $\varphi:A\rightarrow H$ such that for all $a,b \in A$ and $x\in H,$ the followings hold:
	\begin{align}
		\rho_{a}(x)-\rho^\prime_{a}(x)=&~[\varphi(a),x]_{H}, \\
		\psi_{a}(x)-\psi^\prime_{a}(x)=&~x\cdot_{H}\varphi(a), \\
		\phi_{a}(x)-\phi^\prime_{a}(x)=&~\varphi(a)\cdot_{H}x, \\
		\sigma(a,b)-\sigma^{\prime}(a,b)
		=&~\rho^\prime_{a}(\varphi(b))-\rho^\prime_{b}(\varphi(a))-\varphi([a,b]_{A})+[\varphi(a),\varphi(b)]_{H},\\
		\omega(a,b)-\omega^{\prime}(a,b)=&~\phi^\prime_{a}(\varphi(b))+\psi^\prime_{b}(\varphi(a))-\varphi(a\cdot_{A} b)+\varphi(a)\cdot_{H}\varphi(b).
	\end{align}
	
	$\bullet$ Let ${\mathcal H}^2_{nab} (A , H )$ denote the set of equivalence classes of non-abelian 2-cocycles. This set is called the second non-abelian cohomology group of the post-Lie algebra $A$ with values in $H$.
\end{definition}

\begin{definition}
	Let $A $ and $H$ be two post-Lie algebras. A {\bf  non-abelian extension} of $A $ by $H $ is a post-Lie algebra ${E}$ equipped with a  short exact sequence of post-Lie algebras
	\begin{align}\label{abelian-diagram}
		\xymatrix{
			0 \ar[r] & H  \ar[r]^{i} & {E}  \ar[r]^{p} & A  \ar[r] & 0.
		}
	\end{align}
	We often denote a non-abelian extension as above simply by ${E} $ when the underlying short exact sequence is clear from the context.
\end{definition}

\begin{definition}\label{defn-nab-equiv}
	Let \( E \) and \( E' \) be two non-abelian extensions of \( A \) by \( H \). These two extensions are considered \textbf{equivalent} if there exists a homomorphism \(\theta: E \rightarrow E'\) of post-Lie algebras that makes the following diagram commutative
	\begin{align}\label{abelian-equiv}
		\xymatrix{
			0 \ar[r] & H   \ar@{=}[d] \ar[r]^{i} & {E}  \ar[d]^\theta \ar[r]^{p} & A   \ar@{=}[d] \ar[r] & 0 \\
			0 \ar[r] & H  \ar[r]_{i'} & {E}'  \ar[r]_{p'} & A  \ar[r] & 0.
		}
	\end{align}
	We denote by $\mathrm{Ext}_{nab}(A , H )$ the set of all equivalence classes of non-abelian extensions of $A $ by $H $.
\end{definition}

Consider a non-abelian extension $E$ of a post-Lie algebra $A$ by $H$.  A linear section of $p$ is a linear map $s: A \rightarrow {E}$ that satisfies $p \circ s = \mathrm{id}_A$. A linear  section of $p$ always exists. Given a section $s : A \rightarrow {E}$, we define maps
$\rho: A \rightarrow \mathrm{End}(H), a\mapsto \rho_a$, $\psi : A \rightarrow \mathrm{End}(H), a\mapsto \psi_a$, $\phi : A \rightarrow \mathrm{End}(H), a\mapsto \phi_a$, $\sigma :  A\times A \rightarrow H$ and $\omega :  A\times A \rightarrow H$  by
\begin{align}\label{three-maps}
	\begin{cases}
		\rho_a(x):=[s(a),x]_{{E}},\\
		\psi_a(x):=x\cdot_{{E}} s(a),\\
		\phi_a(x):=s(a)\cdot_{{E}} x,\\
		\sigma(a,b):=[s(a),s(b)]_{{E}}-s[a,b]_{A},\\
		\omega(a,b):=s(a)\cdot_{{E}} s(b)-s(a\cdot_{A} b).
	\end{cases}
\end{align}

\begin{proposition}\label{thm:2-cocylce}
	With the above notations, $\rho,\psi,\phi$  are actions of  $A$ on $H$.
	Then $(\rho,\psi,\phi,\sigma,\omega)$ is a non-abelian $2$-cocycle of $A$ with coefficients in $H$.
	Moreover, equivalent non-abelian extensions give equivalent 2-cocycles.
\end{proposition}

\begin{proof}
	First, we show that $\rho,\psi,\phi,\sigma$ and $\omega$ are well defined.
	
 Since $\ker p_{} \cong H_{}$, then for $x\in H_{}$, we have $p_{}(x)=0$.
	By the fact that $p$ is a post-Lie algebra homomorphism, we get
	$$p_{}[s(a),x]_{E}=[p_{}s(a),p_{}(x)]_{E}=[p_{}s(a), 0]_{E}=0,$$
	$$p_{}(x\cdot_{{E}} s(a))=p(x)\cdot_{{E}} ps(a)=0\cdot_{{E}} ps(a)=0,$$
	$$p_{}(s(a)\cdot_{{E}} x)=ps(a)\cdot_{{E}} (p(x)=ps(a)\cdot_{{E}} 0=0.$$	
	Thus, $[s(a), x]_E$, $x \cdot_E s(a)$ and $s(a) \cdot_E x \in \ker p \cong H$.
  This shows that $\rho$, $\psi$ and $\phi$ are well-defined.	
	
	Since $p$ is a post-Lie algebra homomorphism, we have
	\begin{eqnarray*}
		p\sigma(a,b)&=&p([s(a),s(b)]_E)-p(s[a,b]_A)\\
		&=&([ps(a),ps(b)]_E)-(ps[a,b]_A)=[a,b]-[a,b]=0,
	\end{eqnarray*}
	\begin{eqnarray*}
		p\omega(a,b)&=&p(s(a)\cdot_{{E}} s(b))-ps(a\cdot_{A} b)\\
		&=&ps(a)\cdot_{{E}} ps(b)-ps(a\cdot_{A} b)=a\cdot b-a\cdot b=0.
	\end{eqnarray*}
	Thus $\sigma(x,y)\in \ker p=H$ and $\omega(x,y)\in \ker p=H$. This show that $\sigma$ and $\omega$ are well defined.

	Next, we show that $(\rho,\psi,\phi,\sigma,\omega)$ is a non-abelian $2$-cocycle.
Since $A$ and $E$ are post-Lie algebras, i.e., they satisfy conditions \eqref{def3}--\eqref{def13}, we will only verify condition \eqref{def4} here.
By the equality
$$s(a)\cdot_E [s(b),s(c)]_E-[s(a)\cdot_E s(b), s(c)]_E+[s(b),s(a)\cdot_E s(c)]_E=0,$$	
we have	
\begin{eqnarray*}
&&s(a)\cdot_E [s(b),s(c)]_E-[s(a)\cdot_E s(b), s(c)]_E+[s(b),s(a)\cdot_E s(c)]_E\\
&=&s(a)\cdot_E (s[b,c]_{A}+\sigma(b,c))
-[s(a\cdot_{A} b)+\omega(a,b), s(c)]_E
+[s(b),s(a\cdot_{A} c)+\omega(a,c)]_E\\
&=&s(a\cdot_{A} [b,c]_{A})+\omega(a,[b,c]_{A})+\phi_a(\sigma(b,c))-s[(a\cdot_{A} b,c)]_{A}-\sigma(s(a\cdot_{A} b,c)\\
&&-\rho_c(\omega(a,b))+s[b,a\cdot_{A} c]_{A}+\omega(b,a\cdot_{A} c)+\rho_b(\omega(a,c))\\
&=&\omega(a,[b,c]_{A})+\phi_a(\sigma(b,c))
-\sigma(s(a\cdot_{A} b,c)-\rho_c(\omega(a,b))
+\omega(b,a\cdot_{A} c)+\rho_b(\omega(a,c))\\
&=&0.
	\end{eqnarray*}
Hence, we obtain the non-abelian $2$-cocycle condition (A10).
Conditions (A1)--(A9) have been verified via similar computations, and therefore $(\rho,\psi,\phi,\sigma,\omega)$ is a non-abelian $2$-cocycle.
		
Suppose that $E$ and $E'$ are equivalent non-abelian extensions, and $\theta:E\rightarrow E'$ is the Lie algebra homomorphism satisfying $\theta i=i'$, $p'\theta=p$.
Choose linear sections $s$ and $s'$ of $p$ and $p'$ respectively. Then we have
$$p'\theta s(a)=ps(a)=a=p's'(a).$$
Thus $\varphi(a)=\theta{}s(a)-s'(a)\in \Ker p'_{}=H_{}$. Since $\theta|_{H}=\id_{H}$, we obtain
$$[s'(a), x]_{E'} + [\varphi(a), x] = [\theta s(a), x]_{E'} = \theta([s(a), x]_E) = [s(a), x]_E.$$
Therefore,
$$\rho'(a)(x)+[\varphi(a),x]=\rho(a)(x).$$
Similarly, we have $\psi_{a}(x)+x\cdot_{H}\varphi(a)=\psi^\prime_{a}(x)$ and $\phi_{a}(x)+\varphi(a)\cdot_{H}x=\phi^\prime_{a}(x),$
which implies that equivalent non-abelian extensions give equivalent $\rho,\psi,\phi$.

At last, we compute
\begin{eqnarray*}
	&&\sigma(a,b)-\sigma^{\prime}(a,b)\\
	&=&	\theta([s(a),s(b)]_{E}-s[a,b]_{A})
	-([s^\prime(a),s^\prime(b)]_{E'}-s^\prime[a,b]_{A})\\
	&=&[	\theta s(a),	\theta s(b)]_{E'}-	\theta s[a,b]_{A}
	-[s^\prime(a),s^\prime(b)]_{E'}+s^\prime[a,b]_{A}\\
	&=&[(s^{\prime}+\varphi)(a),(s^{\prime}+\varphi)(b)]_{E'}-(s^{\prime}+\varphi)[a,b]_{A}-[s^\prime(a),s^\prime(b)]_{E'}+s^\prime[a,b]_{A}\\
	&=&\rho^\prime_a(\varphi(b))-\rho^\prime_{b}(\varphi(a))-\varphi([a,b]_{A})+[\varphi(a),\varphi(b)]_{H},
\end{eqnarray*}
\begin{eqnarray*}
	&&\omega(a,b)-\omega^{\prime}(a,b)\\
&=&\theta(s(a)\cdot _{{E}} s(b)-s(a\cdot_{A} b))-(s^\prime(a)\cdot_{{E'}} s^\prime(b)-s^\prime (a\cdot_{A} b))\\
&=&\theta s(a)\cdot _{{E'}} \theta s(b)-\theta s(a\cdot_{A} b)-s^\prime(a)\cdot_{{E'}} s^\prime(b)+s^\prime (a\cdot_{A} b)\\
&=&(s^{\prime}+\varphi)(a)\cdot_{{E'}}(s^{\prime}+\varphi)(b)-(s^{\prime}+\varphi)(a\cdot_{A} b)-s^\prime(a)\cdot_{{E'}} s^\prime(b)+s^\prime (a\cdot_{A} b)\\
&=&\phi^\prime_{b}(\varphi(a))+\psi^\prime_{a}(\varphi(b))-\varphi(a\cdot_{A} b)+\varphi(a)\cdot_{H}\varphi(b).
\end{eqnarray*}
Therefore the cohomology class ${\mathcal H}^2_{nab} (A , H )$  does not depend on the choice of $s$.
\end{proof}

 \begin{theorem}
 	Let $A $ and $H $ be two post-Lie algebras. Then the set of equivalence classes of non-abelian extensions of $A $ by $H $ are classified by ${\mathcal H}^2_{nab} (A , H )$.
 \end{theorem}
\begin{proof}
	Let $A $ and $H $ be two post-Lie algebras. Then we define a well-defined map
\begin{align*}
	\Upsilon : \mathrm{Ext}_{nab}(A , H ) \longrightarrow {\mathcal H}^2_{nab} (A , H ).
\end{align*}
as follows.

Let ${E} $ and ${E}' $ be two equivalent extensions of $A $ by $H $ (see Definition \ref{defn-nab-equiv}) and let $s: A \rightarrow {E}$ be a section of the map $p$. Then we get 	
$$p' \circ (\varphi \circ s) = (p' \circ \varphi) \circ s = p \circ s = \mathrm{id}_A$$
which shows that $s' = \varphi \circ s : A \rightarrow {E}'$ is a section of the map $p'$.

Let $(\rho',\phi',\psi',\sigma',\omega')$ be the non-abelian $2$-cocycle corresponding to the extension ${E}' $ and section $p'$.
Then for all $a, b \in A$ and $x \in H$, we have
	\begin{align*}
		\rho_a' (x) = [s'(a),x ]_{{E}'} =~& [ \varphi \circ s(a),x]_{{E}'} \\
		=~& [\varphi (s(a)),\varphi (x)]_{{E}'} \\
		=~& \varphi ([ s(a),x]_{{E}}) = \varphi \big( \rho_a (x) \big) = \rho_a (x),
	\end{align*}
	\begin{align*}
		\psi_a' (x) = x\cdot_{{E}' }s'(a) =~& x\cdot_{{E}'} \varphi \circ s(a) \\
		=~&\varphi (x)\cdot_{{E}'} \varphi (s(a))  \\
		=~& \varphi (x\cdot_{{E}}s(a)) = \varphi \big( \psi_a (x) \big) = \psi_a (x),
	\end{align*}
	\begin{align*}
		\phi_a' (x) = s'(a)\cdot_{{E}' }x =~& \varphi \circ s(a)\cdot_{{E}'} x \\
		=~& \varphi (s(a))\cdot_{{E}'} \varphi (x) \\
		=~& \varphi (s(a)\cdot_{{E}}x) = \varphi \big( \phi_a (x) \big) = \phi_a (x),
	\end{align*}
	\begin{align*}
	\sigma' (a,b) =~& [s' (a), s'(b)]_{{E}'} - s' [a,b]_A \\
	=~& [\varphi \circ s (a), \varphi \circ s (b) ]_{{E}'} - \varphi \circ s [a, b]_A \\
	=~& \varphi \big(   [s (a), s (b) ]_{{E}} - s [a,b]_A  \big) \\
	=~& \sigma (a,b ) \quad (\text{as } \varphi|_H = \mathrm{id}_H),
\end{align*}
\begin{align*}
	\omega' (a,b) =~& s' (a)\cdot_{{E}'} s'(b) - s'( a\cdot_A b) \\
	=~& \varphi \circ s (a)\cdot_{{E}'} \varphi \circ s (b)  - \varphi \circ s( a\cdot_A b) \\
	=~& \varphi \big(   s (a)\cdot_{{E}} s (b)  - s (a\cdot_A b ) \big) \\
	=~& \omega (a,b ) \quad (\text{as } \varphi|_H = \mathrm{id}_H).
\end{align*}
Thus we obtain that equivalent extensions give the equivalent non-abelian 2-cocycles $(\rho,$ $\phi,\psi,$ $\sigma,$ $\omega)$ and $(\rho',\phi',\psi',\sigma',\omega')$.
This shows that the map $\Upsilon : \mathrm{Ext}_{nab} (A , H ) \rightarrow {\mathcal H}^2_{nab} (A , H )$ assigning an equivalent class of extensions to the class of corresponding non-abelian $2$-cocycle is well-defined.

Conversely,  we define the map
\begin{align*}
	\Omega : {\mathcal H}^2_{nab} (A , H ) \longrightarrow \mathrm{Ext}_{nab}(A , H ).
\end{align*}
Let $(\rho, \phi, \psi, \sigma, \omega)$ be a non-abelian 2-cocycle. Consider the vector space $E$ with the product defined by (\ref{product1}) and (\ref{product2}).
Using the conditions (A1)--(A10), it can be easily verify that the the product satisfies the conditions \eqref{def3} and \eqref{def4} of the definition of post-Lie algebra. Note that we have denoted this post-Lie algebra by $A\ltimes_{\sigma, \omega} H$. Let $(\rho',\phi',\psi',\sigma',\omega')$ be another $2$-cocycle equivalent to $(\rho,\phi,\psi,\sigma,\omega)$.
Let $A\ltimes_{\sigma, \omega} H$ and $A'\ltimes_{\sigma', \omega'} H'$ be the post-Lie algebras induced by the $2$-cocycles $(
\rho,\phi,\psi,\sigma,\omega)$ and $(\rho',\phi',\psi',\sigma',\omega')$ respectively. We define a map
\begin{align*}
\Theta : {E} \rightarrow {E'} ~\text{ by } \Theta ((a,x)):= (a , \varphi (a) + x),~\text{for } (a, x) \in {E}.
\end{align*}
Then, it is easy to verify that $\Theta$ is a homomorphism of post-Lie algebras. This demonstrates that the map $\Theta$ defines an equivalence between the non-abelian extensions $E$ and $E'$.  Hence there is a well-defined map $\Omega : {\mathcal H}^2_{nab} (A , H ) \rightarrow \mathrm{Ext}_{nab} (A , H )$.
\end{proof}

\section{Automorphisms}
In the context of a non-abelian extension of post-Lie algebras, we consider the inducibility of a pair of post-Lie algebra automorphisms. We construct the Wells type  short exact sequence that connects  automorphism groups and the second non-abelian cohomology group.

\subsection{Inducibility for automorphisms}
In this section, our main results, Theorems \ref{thm-inducibility} and \ref{thm-inducibility-2}, establish the necessary and sufficient conditions for a pair of automorphisms to be inducible.
\medskip

Let $A $ and $H$ be two post-Lie algebras, and consider the non-abelian extension of post-Lie algebras given by the exact sequence
\begin{align*}
\xymatrix{
0 \ar[r] & H  \ar[r]^{i} & {E}  \ar[r]^{p} & A  \ar[r] & 0.
}
\end{align*}

Define $\mathrm{Aut}_H ({E} )$ as the set of all automorphisms $\gamma \in \mathrm{Aut}({E} )$ that satisfy $\gamma|_H \subset H$. For any $\gamma \in \mathrm{Aut}_H ({E} )$, it follows that $\gamma|_H \in \mathrm{Aut}(H )$. Additionally, we can define a map $\overline{\gamma} : A \rightarrow A$ by
$$\overline{\gamma} (a) := p \gamma s (a), \text{ for } a \in A.$$
Assume that \( s_{1} \) and \( s_{2} \) are two distinct sections of \( p \). Given
$ p s_{1}(x) - p s_{2}(x) = 0, $
we can infer that
$ s_{1}(x) - s_{2}(x) \in \mathrm{ker } \, p \cong H. $
This further implies that
$ \gamma(s_{1}(x) - s_{2}(x)) \in H. $
As a result, we obtain
$ p \gamma s_{1}(x) = p \gamma s_{2}(x).$
Therefore, \( \bar{\gamma} \) is independent of the choice of sections.
Consider the space $A$, which can be viewed as a subspace of $E$ through the section $s$. Notably, $E \cong H \oplus s(A)$. Since $\gamma$ is an automorphism on $E$ that preserves the space $H$, it consequently preserves the subspace $A$ as well. Therefore, the map $\overline{\gamma} = p \gamma s$ is bijective on $A$.

For any $a,b \in A$, we have
\begin{align*}
\overline{\gamma} ([a,b]_A) = p \gamma (s [a,b]_A) =~& p \gamma ( [s(a), s(b)]_{{E}} - \sigma (a,b) ) \\
=~& p \gamma ( [s(a), s(b)]_{{E}} ) \quad (\text{as } \gamma|_H \subset H \text{ and } p|_H = 0) \\
=~& [p \gamma s (a), p \gamma s (b)]_A  = [\overline{\gamma} (a), \overline{\gamma} (b)]_A,
\end{align*}
\begin{align*}
\overline{\gamma} (a\cdot_A b) = p \gamma (s (a\cdot_A b)) =~& p \gamma ( s(a)\cdot_{{E}}s(b) - \omega (a,b) ) \\
=~& p \gamma ( s(a)\cdot_{{E}} s(b) ) \quad (\text{as } \gamma|_H \subset H \text{ and } p|_H = 0) \\
=~& p \gamma s (a)\cdot_A p \gamma s (b)  = \overline{\gamma} (a)\cdot_A \overline{\gamma} (b).
\end{align*}
This demonstrates that $\overline{\gamma} : A \rightarrow A$ is an automorphism, meaning $\overline{\gamma} \in \mathrm{Aut}(A)$. Consequently, this construction results in a group homomorphism
\begin{align*}
\Phi : \mathrm{Aut}_H ({E} ) \rightarrow \mathrm{Aut}(H ) \times \mathrm{Aut}(A ), ~~ \Phi (\gamma) := (\gamma|_H, \overline{ \gamma}).
\end{align*}
A pair of automorphisms $(\beta, \alpha) \in \mathrm{Aut}(H ) \times \mathrm{Aut}(A )$ is said to be {\bf inducible} if the pair $(\beta, \alpha)$ lies in the image of $\Phi$.
In the following result, we provide a necessary and sufficient condition for the pair $(\beta, \alpha)$ to be inducible.

\begin{theorem}\label{thm-inducibility}
Let $0 \rightarrow H  \xrightarrow{i} {E}  \xrightarrow{p} A  \rightarrow 0$ be a non-abelian extension of post-Lie algebras and $(
\rho,\phi,\psi,\sigma,\omega)$ be the corresponding non-abelian $2$-cocycle induced by a section $s$. A pair $(\beta, \alpha) \in \mathrm{Aut}(H ) \times \mathrm{Aut}(A )$ is inducible if and only if there exists a linear map $\lambda \in \mathrm{Hom}(A, H)$ satisfying the following conditions:
\begin{itemize}
\item[(I)] $[ \lambda (a),\beta (x)]_{{H}}=\beta(\rho_{a}x)-\rho_{\alpha (a)}(\beta (x)),$
\item[(II)] $\beta(x)\cdot_{{H}}\lambda(a)=
    \beta(\psi_{a}x)-\psi_{\alpha(a)}\beta(x),$
\item[(III)] $\lambda(a)\cdot_{{H}}\beta(x)=
    \beta(\phi_{a}x)-\phi_{\alpha(a)}\beta(x),$
\item[(IV)] $\beta(\sigma(a_1,a_2))-\sigma(\alpha(a_1),\alpha(a_2))\\
= \rho_{\alpha (a_1)}( \lambda (a_2))-\rho_{\alpha (a_2)}(\lambda (a_1))+[ \lambda (a_1),\lambda (a_2)]_{{H}}-\lambda[a_1,a_2]_{{A}},$
\item[(V)]$ \beta(\omega(a_1,a_2))-\omega(\alpha(a_1),\alpha(a_2))\\
  =\psi_{\alpha(a_2)}\lambda(a_1)+\phi_{\alpha(a_1)}\lambda(a_2)
  +\lambda(a_1)\cdot_{{H}}\lambda(a_2)-\lambda(a_1\cdot_{{A}}a_2)$,
\end{itemize}
for all $a, a_1, a_2 \in A, x \in H$.
\end{theorem}

\begin{proof}
Suppose the pair $(\beta, \alpha) \in \mathrm{Aut}(H ) \times \mathrm{Aut}(A )$ is inducible. Then, there exists an automorphism $\gamma \in \mathrm{Aut}_{H} ({E} )$ such that $\gamma|_{H} = \beta$ and $p \gamma s = \alpha$. For any $a \in A$, we observe that
$p \big(  (\gamma s - s \alpha)(a)  \big) = \alpha (a) - \alpha (a) = 0.$
This demonstrates that $$(\gamma s - s \alpha)(a) \in \mathrm{ker }~ p = \mathrm{im }~ i \cong H.$$ We define a map $\lambda : A \rightarrow H$ by
$\lambda(a) := (\gamma s - s \alpha)(a), \text{ for } a \in A.$
Then we have
\begin{align*}
\beta(\rho_{a}x)-\rho_{\alpha (a)}(\beta (x))
 =~&\beta ( [ s(a),x]_{{E}})-[s \alpha (a),\beta (x)]_{{E}}\\
=~& [\gamma s (a),\gamma (x)]_{{E}} -[ s \alpha (a),\beta (x)]_{{E}} \\
=~&[\gamma s (a),\beta (x)]_{{E}} -[s \alpha (a),\beta (x)]_{{E}}\\
=~& [\lambda (a),\beta (x)]_H,
\end{align*}
\begin{align*}
\beta(\psi_{a}x)-\psi_{\alpha(a)}\beta(x) =~& \beta (x \cdot_{{E}} s(a)) - \beta (x)\cdot_{{E}}s \alpha (a) \\
=~&\gamma (x) \cdot_{{E}}\gamma s (a) - \beta (x)\cdot_{{E}}s \alpha (a) \\
=~& \beta (x)\cdot_{{E}}\gamma s (a) - \beta (x)\cdot_{{E}} s \alpha (a) \\
=~& \beta (x)\cdot_H \lambda (a),
\end{align*}
\begin{align*}
\beta(\phi_{a}x)-\phi_{\alpha(a)}\beta(x) =~& \beta ( s(a)\cdot_{{E}} x) - s \alpha (a)\cdot_{{E}}\beta (x) \\
=~& \gamma s (a)\cdot_{{E}}\gamma (x) - s \alpha (a)\cdot_{{E}}\beta (x) \\
=~& \gamma s (a)\cdot_{{E}}\beta (x) - s \alpha (a)\cdot_{{E}} \beta (x) \\
=~& \lambda (a)\cdot_H \beta (x).
\end{align*}
Hence we get the conditions (I), (II) and (III). Next, we observe that
\begin{align*}
& \rho_{\alpha (a_1)}( \lambda (a_2))-\rho_{\alpha (a_2)}(\lambda (a_1))+[ \lambda (a_1),\lambda (a_2)]_{{H}}-\lambda[a_1,a_2]_{{A}} \\
&=  [ s \alpha (a_1),(\gamma s - s \alpha)(a_2)]_{{E}}-[s \alpha (a_2) ,(\gamma s - s \alpha)(a_1)]_{{E}}  \\
 &- (\gamma s - s \alpha)([a_1, a_2]_A)+ [(\gamma s - s \alpha)(a_1), (\gamma s - s \alpha)(a_2)]_H \\
&= [ s \alpha (a_1),\gamma s (a_2)]_{{E}}- [ s \alpha (a_1),s \alpha(a_2)]_{{E}}-[ s \alpha (a_2),\gamma s (a_1)]_{{E}} \\
 & + [ s \alpha (a_2),s \alpha (a_1) ]_{{E}}- \gamma s ([a_1, a_2]_A) + s [\alpha (a_1), \alpha(a_2)]_A + [\gamma s (a_1), \gamma s (a_2)]_{{E}}\\
& ~~  - [\gamma s (a_1), s \alpha (a_2)]_{{E}} - [s \alpha (a_1), \gamma s (a_2)]_{{E}} + [s \alpha (a_1), s \alpha(a_2)]_{{E}} \\
&= \big( [\gamma s (a_1), \gamma s (a_2)]_{{E}} - \gamma s ([a_1, a_2]_A) \big) - [s \alpha (a_1), s \alpha (a_2)]_{{E}} + s [\alpha(a_1), \alpha (a_2)]_A \\
&= \gamma \big(  [s(a_1), s(a_2)]_{{E}} - s [a_1, a_2]_A \big)  - \big( [s\alpha (a_1) , s \alpha(a_2)]_{{E}} - s [\alpha (a_1), \alpha(a_2)]_A  \big) \\
&= \beta (\sigma (a_1, a_2))- \sigma (\alpha (a_1), \alpha (a_2)),
\end{align*}
\begin{align*}
&\psi_{\alpha(a_2)}\lambda(a_1)+\phi_{\alpha(a_1)}\lambda(a_2)
  +\lambda(a_1)\cdot_{{H}}\lambda(a_2)-\lambda(a_1\cdot_{{A}}a_2)\\
&=(\gamma s - s \alpha)(a_1)\cdot_{{E}}s(\alpha(a_2))+s(\alpha(a_1))\cdot_{{E}}(\gamma s - s \alpha)(a_2)\\
&+(\gamma s - s \alpha)(a_1)\cdot_{{H}}(\gamma s - s \alpha)(a_2)-(\gamma s - s \alpha)(a_1\cdot_{{A}}a_2)\\
&=\gamma s (a_1)\cdot_{{E}}s(\alpha(a_2))- s \alpha(a_1)\cdot_{{E}}s(\alpha(a_2))+s(\alpha(a_1))\cdot_{{E}}\gamma s(a_2)-s(\alpha(a_1))\cdot_{{E}}s \alpha(a_2)\\
&+\gamma s (a_1)\cdot_{{H}}\gamma s (a_2)-\gamma s (a_1)\cdot_{{H}} s \alpha(a_2)- s \alpha(a_1)\cdot_{{H}}\gamma s(a_2)+ s \alpha(a_1)\cdot_{{H}}s \alpha(a_2)\\
&-\gamma s(a_1\cdot_{{A}}a_2)+s \alpha(a_1\cdot_{{A}}a_2)\\
&=\gamma(s(a_1)\cdot_{{E}} s(a_2)-\gamma s(a_1\cdot_{A} a_2))-s(\alpha(a_1))\cdot_{{E}} s(\alpha(a_2))+s(\alpha(a_1)\cdot_{A} \alpha(a_2))\\
&=\beta(s(a_1)\cdot_{{E}} s(a_2)-s(a_1\cdot_{A} a_2))-s(\alpha(a_1))\cdot_{{E}} s(\alpha(a_2))+s(\alpha(a_1)\cdot_{A} \alpha(a_2))\\
&=\beta(\omega(a_1,a_2))-\omega(\alpha(a_1),\alpha(a_2)).
  \end{align*}
This establishes conditions (IV) and (V).

Conversely, assume that there exists a linear map $\lambda \in \mathrm{Hom}(A, H)$ that satisfies conditions (I)--(V). Since $s$ is a section of $p$, any element $e \in E$ can be expressed as $e = x + s(a)$ for some $x \in H$ and $a \in A$.  We define a map $\gamma : {E} \rightarrow {E}$ by
\begin{align}\label{g-defn}
\gamma (e) = \gamma (x + s(a)) := \big(  \beta (x) + \lambda (a) \big) + s (\alpha (a)).
\end{align}
To establish that $\gamma$ is bijective, it is necessary to show that $\gamma (x + s(a)) = 0$. According to equation (\ref{g-defn}), this implies that $s(\alpha(a)) = 0$. Given that both $s$ and $\alpha$ are injective, it follows that $a = 0$. Substituting this result back into equation (\ref{g-defn}), we find that $\beta(x) = 0$, which further implies that $x = 0$. This implies $x + s(a) = 0$, establishing that $\gamma$ is injective.  Finally, for all element $e = x + s(a) \in {E}$, we consider the element $$e' = \big( \beta^{-1}(x) - \beta^{-1}\lambda \alpha^{-1} (a)  \big) + s(\alpha^{-1}(a)) \in {E}.$$ Then we have
\begin{align*}
\gamma (e') =  \big( \beta \beta^{-1} (x) - \beta \beta^{-1} \lambda \alpha^{-1} (a) + \lambda \alpha^{-1} (a) \big) + s(\alpha \alpha^{-1}(a))  = x + s(a) = e.
\end{align*}
This demonstrates that $\gamma$ is surjective, confirming our claim. Next, for all two elements $e_1 = x_1 + s(a_1)$ and $e_2 = x_2 + s(a_2)$ in the vector space ${E}$, we have
\begin{align*}
&[\gamma (e_1), \gamma (e_2)]_{{E}} \\
&= [\gamma (x_1 + s(a_1)), \gamma (x_2 + s(a_2)) ]_{{E}} \\
&= [ \beta (x_1) + \lambda (a_1) + s (\alpha (a_1)) ~,~ \beta (x_2) + \lambda (a_2) + s (\alpha (a_2))  ]_{{E}}  \\
&=[ \beta (x_1) ~,~ \beta (x_2)  ]+[ \beta (x_1)  ~,~ \lambda (a_2)  ]+[ \beta (x_1)  ~,~  s (\alpha (a_2))  ]+ [  \lambda (a_1)  ~,~ \beta (x_2) ]\\
&\quad+ [  \lambda (a_1),\lambda (a_2)  ]+ [  \lambda (a_1),~ s (\alpha (a_2))  ] +[ s (\alpha (a_1)),~ \beta (x_2)  ]+[ s (\alpha (a_1)),~  \lambda (a_2)] \\
&\quad+[ s (\alpha (a_1)),~ s (\alpha (a_2))  ]  \\
&=s[\alpha (a_1),\alpha (a_2)]+\sigma(\alpha (a_1),\alpha (a_2))+[\beta (x_1) ,\beta (x_2) ]_{{H}}+[\beta (x_1),\lambda (a_2)]_{{H}}+[ \lambda (a_1),\beta (x_2)]_{{H}}\\
&\quad +[ \lambda (a_1),\lambda (a_2)]_{{H}}-\rho_{\alpha (a_2)}(\beta (x_1))-\rho_{\alpha (a_2)}(\lambda (a_1))+ \rho_{\alpha (a_1)}(\beta (x_2))+ \rho_{\alpha (a_1)}( \lambda (a_2))\\
&= s(\alpha[a_1,a_2]_{A})+\beta(\sigma(a_1,a_2))+\beta[x_1,x_2]_{{H}}-\beta(\rho_{a_2}x_1)+\beta(\rho_{a_1}x_2)
+\lambda[a_1,a_2]_{{A}}\\
&=\gamma \big(\sigma(a_1,a_2)+s[a_1,a_2]_{A}+[x_1,x_2]_{{H}}-\rho_{a_2}x_1+\rho_{a_1}x_2\big)\\
&= \gamma \big([s(a_1),s(a_2)]_{{A}}+[x_1,x_2]_{{H}}-\rho_{a_2}x_1+\rho_{a_1}x_2\big) \\
&= \gamma \big(  [x_1 + s(a_1), x_2 + s(a_2)]_{{E}} \big) = \gamma ([e_1, e_2]_{{E}}),
\end{align*}
\begin{align*}
&\gamma (e_1)\cdot_{{E}} \gamma (e_2) \\
&= \gamma (x_1 + s(a_1))\cdot_{{E}} \gamma (x_2 + s(a_2)) \\
&= (\beta (x_1) + \lambda (a_1) + s (\alpha (a_1))) ~\cdot_{{E}}~ (\beta (x_2) + \lambda (a_2) + s (\alpha (a_2)))    \\
&=s(\alpha(a_1)\cdot_{{A}}\alpha(a_2))+\omega(\alpha(a_1),\alpha(a_2))+\beta(x_1)\cdot_{{H}}\beta(x_2)+\beta(x_1)\cdot_{{H}}\lambda(a_2)\\
&\quad+\lambda(a_1)\cdot_{{H}}\beta(x_2)+\lambda(a_1)\cdot_{{H}}\lambda(a_2)+\psi_{\alpha(a_2)}\beta(x_1)+\psi_{\alpha(a_2)}\lambda(a_1)+\phi_{\alpha(a_1)}\beta(x_2)\\
&\quad+\phi_{\alpha(a_1)}\lambda(a_2)\\
&=s(\alpha(a_1\cdot_{A} a_2))+\beta(\omega(a_1,a_2))+\beta(x_1\cdot_{{H}}x_2)+\beta(\psi_{a_2}x_1)+\beta(\phi_{a_1}x_2)
+\lambda(a_1\cdot_{{A}}a_2) \\
&=\gamma(\omega(a_1,a_2)+s(a_1\cdot_{A} a_2)+x_1\cdot_{{H}}x_2+\psi_{a_2}x_1+\phi_{a_1}x_2) \\
&=\gamma(s(a_1)\cdot_{{A}}s(a_2)+x_1\cdot_{{H}}x_2+\psi_{a_2}x_1+\phi_{a_1}x_2) \\
&= \gamma \big( ( x_1 + s(a_1))\cdot_{{E}}(x_2 + s(a_2) )\big) = \gamma (e_1\cdot_{{E}}e_2).
\end{align*}
Therefore, $\gamma : {E} \rightarrow {E}$ is a post-Lie algebra homomorphism, proving that $\gamma$ is an automorphism of the post-Lie algebra ${E}$. In other words, $\gamma \in \mathrm{Aut}({E})$.
Given that $\gamma|_H \subset H$,  it follows that $\gamma \in \mathrm{Aut}_{H} ({E} )$. Consider \(x \in H\) and \(a \in A\). We have
$$\gamma (x) =\gamma (x + s(0)) = \beta (x) $$ and
$$p \gamma s(a) = p \gamma (0 + s(a)) = p (\lambda (a) + s(\alpha(a))) = ps (\alpha (a)) = \alpha (a).
$$
Thus, we can conclude that \(\gamma|_H = \beta\) and \(p \gamma s = \alpha\). This establishes that the pair \((\beta, \alpha) \in \mathrm{Aut}(H) \times \mathrm{Aut}(A)\) is inducible.
\end{proof}

Consider the non-abelian extension $0 \rightarrow H \xrightarrow{i} E \xrightarrow{p} A \rightarrow 0$ of $A$ by $H$, with the corresponding non-abelian $2$-cocycle $(\rho, \phi, \psi, \sigma, \omega)$. For any $(\beta, \alpha) \in \mathrm{Aut}(H) \times \mathrm{Aut}(A)$, we define  $(\rho, \phi, \psi, \sigma, \omega)_{(\beta, \alpha)}$  of linear maps
\begin{eqnarray*}
	&&\rho_{( \beta, \alpha)} : A \rightarrow \mathrm{End}(H),\quad\psi_{( \beta, \alpha)} : A \rightarrow \mathrm{End}(H),
	\quad\phi_{( \beta, \alpha)} : A \rightarrow \mathrm{End}(H),\\
	&&\sigma_{(\beta, \alpha)} :  A\times A \rightarrow H, \quad\quad~\omega_{(\beta, \alpha)} :  A\times A \rightarrow H,
\end{eqnarray*}by
\begin{align}\label{aut-2-defi}
 (\rho_{(\beta, \alpha)})_a x = \beta (\rho_{\alpha^{-1}(a)} \beta^{-1}(x)),\\
  (\psi_{(\beta, \alpha)})_a x = \beta (\psi_{\alpha^{-1}(a)} \beta^{-1}(x)),\\
   (\phi_{(\beta, \alpha)})_a x = \beta (\phi_{\alpha^{-1}(a)} \beta^{-1}(x)),\\
   \sigma_{(\beta, \alpha)} (a, b) = \beta \circ \sigma (\alpha^{-1}(a), \alpha^{-1}(b)),\\
   \omega_{(\beta, \alpha)} (a, b) = \beta \circ \omega(\alpha^{-1}(a), \alpha^{-1}(b)).
\end{align}
for $a, b \in A$ and $x \in H$.

\begin{proposition}
The $(
\rho,\psi,\phi_,\sigma,\omega)_{(\beta, \alpha)}$ is a non-abelian $2$-cocycle.
\end{proposition}
\begin{proof}
Given that $(\rho, \phi, \psi, \sigma, \omega)$ is a non-abelian 2-cocycle, we substitute $a$, $b$ and $x$ with $\alpha^{-1}(a)$, $\alpha^{-1}(b)$ and $\beta^{-1}(x)$, respectively, in the conditions of a 2-cocycle. We derive the corresponding conditions for the  $(
\rho,\psi,\phi_,\sigma,\omega)_{(\beta, \alpha)}$. For example, it follows from (A8) that
\begin{align*}
~&\beta (\psi_{\alpha^{-1}[b,c]} \beta^{-1}(x))\\
=~&\beta (\rho_{\alpha^{-1}(b)} \beta^{-1}(\beta (\psi_{\alpha^{-1}(c)} \beta^{-1}(x))))
-\beta (\rho_{\alpha^{-1}(c)} \beta^{-1}(\beta (\psi_{\alpha^{-1}(b)} \beta^{-1}(x))))\\
~&- x \cdot(\beta \circ \sigma (\alpha^{-1}(b), \alpha^{-1}(c))),
\end{align*}
This is same as
\begin{align*}
~&(\psi_{(\alpha,\beta)})_{[b,c]_A} x
=(\rho_{(\alpha,\beta)})_{(b)}  (\psi_{(\alpha,\beta)})_{(c)} x
- (\rho_{(\alpha,\beta)})_{(c)}  (\psi_{(\alpha,\beta)})_{(b)}x
- x \cdot_{{E}} \sigma_{(\alpha,\beta)} (b, c),
\end{align*}
This shows that the condition (A8) follows for the $(\rho,\psi,\phi_,\sigma,\omega)_{(\beta, \alpha)}$.
The other conditions can be obtained similarly.
\end{proof}
With the above notations, Theorem \ref{thm-inducibility} can be rephrased as follows.
\begin{theorem}\label{thm-inducibility-2}
Let $0 \rightarrow H  \xrightarrow{i} {E}  \xrightarrow{p} A  \rightarrow 0$ be a non-abelian extension of post-Lie algebras. A pair of automorphisms $(\beta, \alpha) \in \mathrm{Aut}(H ) \times \mathrm{Aut}(A )$ is inducible if and only if the non-abelian $2$-cocycles $(
\rho,\phi,\psi,\sigma,\omega)$ and $(\rho,\psi,\phi_,\sigma,\omega)_{(\beta, \alpha)}$ are equivalent.
This equivalence means that they correspond to the same element in ${\mathcal H}^2_{nab}(A , H )$.
\end{theorem}

\begin{proof}
Let the pair $(\beta, \alpha) \in \mathrm{Aut}(H) \times \mathrm{Aut}(A )$ be inducible. By Theorem \ref{thm-inducibility}, there exists a linear map $\lambda \in \mathrm{Hom} (A, H)$ satisfying (I)--(V). To apply these conditions, we replace $a, b, x$ with $\alpha^{-1}(a), \alpha^{-1}(b), \beta^{-1}(x)$, respectively. We obtain
\begin{align*}
(\rho_{(\beta, \alpha)})_a x - \rho_a x =~& [\lambda \alpha^{-1} (a),x]_H, \\
(\psi_{(\beta, \alpha)})_a x - \psi_a x =~& x\cdot_H \lambda \alpha^{-1} (a), \\
(\phi_{(\beta, \alpha)})_a x - \phi_a x =~& \lambda \alpha^{-1} (a)\cdot_H x, \\
\sigma_{(\beta, \alpha)} (a,b) - \sigma (a,b) =~&  \rho_a \lambda \alpha^{-1} (b) -\rho_b \lambda \alpha^{-1} (a) - \lambda \alpha^{-1} ([a,b]_A) + [\lambda \alpha^{-1} (a), \lambda \alpha^{-1}(b)]_H,\\
\omega_{(\beta, \alpha)} (a,b) - \omega (a,b) =~& \psi_b \lambda \alpha^{-1} (a) - \phi_a \lambda \alpha^{-1} (b) - \lambda \alpha^{-1} (a\cdot_A b) + \lambda \alpha^{-1} (a)\cdot_H \lambda \alpha^{-1}(b).
\end{align*}
This shows that $(\rho,\phi,\psi,\sigma,\omega)$ and $(\rho,\psi,\phi_,\sigma,\omega)_{(\beta, \alpha)}$ are equivalent and the equivalence is given by the map $\lambda \alpha^{-1} : A \rightarrow H$.

Conversely, suppose that the non-abelian $2$-cocycles $(\rho,\phi,\psi,\sigma,\omega)$ and $(\rho,\psi,\phi_,\sigma,\omega)_{(\beta, \alpha)}$ are equivalent, with the equivalence given by a map $\varphi : A \rightarrow H$. It can be easily verified that the map  $\lambda := \varphi \alpha : A \rightarrow H$ satisfies the conditions (I), (II), (III), (IV) and (V) of Theorem \ref{thm-inducibility}. Therefore, the pair $(\beta, \alpha) \in \mathrm{Aut} (H) \times \mathrm{Aut} (A )$ is inducible.
The proof is completed.
\end{proof}

\subsection{Wells exact sequence}

In this section, we define the Wells map associated with a non-abelian extension of post-Lie algebras and we find that it fits into a short exact sequence.

Let $0 \rightarrow H  \xrightarrow{i} {E}  \xrightarrow{p} A  \rightarrow 0$ be a non-abelian extension of the post-Lie algebra $A $ by $H $. This extension induces a corresponding non-abelian $2$-cocycle  $(
\rho,\phi,\psi,\sigma,\omega)$ through a fixed section $s$. We define a set map $\mathcal{W} : \mathrm{Aut}(H ) \times \mathrm{Aut} (A ) \rightarrow {\mathcal H}^2_{nab} (A , H )$ as
\begin{align}\label{wells-m}
\mathcal{W}(\beta, \alpha) = [(
\rho,\psi,\phi_,\sigma,\omega)_{(\beta, \alpha)} - (
\rho,\phi,\psi,\sigma,\omega)]
\end{align}
the equivalence class of $(
\rho,\psi,\phi_,\sigma,\omega)_{(\beta, \alpha)} - (
\rho,\phi,\psi,\sigma,\omega)$.
It is referred to as the {\bf Wells map} and denoted by $\mathcal{W}$.

\medskip

\begin{proposition}\label{wells-s-ind}
The Wells map $\mathcal{W}$ does not depend on the choice of sections.
\end{proposition}

\begin{proof}
Let $s'$ be any other section with the induced non-abelian $2$-cocycle $(
\rho',\phi',\psi',\sigma',\omega')$. The non-abelian $2$-cocycles $(\rho, \phi, \psi, \sigma, \omega)$ and $(\rho', \phi', \psi', \sigma', \omega')$ are equivalent by the map $\varphi$ as defined in Definition \ref{equivalent}. On the other hand, for all $a,b \in A$ and $x \in H$, we get
\begin{align*}
(\rho_{(\beta, \alpha)})_a (x) - (\rho'_{(\beta, \alpha)})_a (x) =~& \beta (\rho_{\alpha^{-1} (a)} \beta^{-1}(x)) - \beta (\rho'_{\alpha^{-1} (a)} \beta^{-1}(x)) \\
=~& \beta [   \varphi \alpha^{-1} (a),\beta^{-1}(x)]_H = [ \beta \varphi \alpha^{-1} (a),x]_H,
\end{align*}
\begin{align*}
\sigma_{(\beta, \alpha)} (a,b) - \sigma'_{(\beta, \alpha)} (a,b) =~& (\rho'_{(\beta, \alpha)})_a (\beta \varphi \alpha^{-1} (b)) - (\rho'_{(\beta, \alpha)})_b (\beta \varphi \alpha^{-1} (a)) \\
~& \qquad - \beta \varphi \alpha^{-1} ([a,b]_A) + [\beta \varphi \alpha^{-1} (a), \beta \varphi \alpha^{-1} (b)]_H.
\end{align*}
The proof for other conditions follows a similar approach.
This demonstrates that
$$(\rho, \phi, \psi, \sigma, \omega)_{(\beta, \alpha)}~~\text{and}~~ (\rho',\psi',\phi',\sigma',\omega')_{(\beta, \alpha)}$$
are equivalent, with the equivalence given by $\beta \varphi \alpha^{-1}$. Combining the results above, we have that the non-abelian $2$-cocycles
$$(\rho, \phi, \psi, \sigma, \omega)_{(\beta, \alpha)} - (\rho,\phi,\psi,\sigma,\omega)$$
and
$$(\rho',\psi',\phi',\sigma',\omega')_{(\beta, \alpha)} - (\rho',\phi',\psi',\sigma',\omega')$$
are equivalent to each other and their equivalence is given by $\beta \varphi \alpha^{-1} - \varphi$. Hence, they corresponds to the same element in ${\mathcal H}^2_{nab} (A , H )$. This completes the proof.
\end{proof}

\begin{remark}
Theorem \ref{thm-inducibility-2} implies that a pair of automorphisms $(\beta, \alpha) \in \mathrm{Aut}(H) \times \mathrm{Aut}(A)$ is inducible if and only if $\mathcal{W}(\beta, \alpha)$ is trivial. In other words, $\mathcal{W}(\beta, \alpha)$ serves as an obstruction to the inducibility of the pair $(\beta, \alpha)$. Consequently, if the Wells map $\mathcal{W}$ is trivial, any pair of post-Lie algebra automorphisms $(\beta, \alpha) \in \mathrm{Aut}(H) \times \mathrm{Aut}(A)$ is inducible.
\end{remark}

\begin{theorem}\label{wells-seq}
Let $0 \rightarrow H  \xrightarrow{i} {E}  \xrightarrow{p} A  \rightarrow 0$ be a non-abelian extension of post-Lie algebras. Then there is an exact sequence
\begin{align*}
1 \rightarrow \mathrm{Aut}_H^{H, A} ({E} ) \xrightarrow{\iota} \mathrm{Aut}_H ({E} ) \xrightarrow{\Phi} \mathrm{Aut}(H ) \times \mathrm{Aut}(A ) \xrightarrow{\mathcal{W}} {\mathcal H}^2_{nab}(A , H ).
\end{align*}
Here $\mathrm{Aut}_H^{H, A} ({E} ) = \{ \gamma \in \mathrm{Aut} ({E} ) ~|~ \Phi(\gamma) = (\mathrm{id}_H, \mathrm{id}_A) \}$.
\end{theorem}

\begin{proof}
Given that the inclusion map $\iota : \mathrm{Aut}_H^{H, A} ({E}) \rightarrow \mathrm{Aut}_H ({E})$ is an injection, the sequence is exact at the first term.
Consider an element $\gamma \in \mathrm{ker} (\Phi)$. This implies that $\gamma \in \mathrm{Aut}_H({E})$ and satisfies $\gamma|_H = \mathrm{id}_H$ and $p \gamma s = \mathrm{id}_A$. Therefore, $\gamma \in \mathrm{Aut}_H^{H, A} ({E})$. Conversely, if $\gamma \in \mathrm{Aut}_H^{H, A} ({E})$, then $\gamma \in \mathrm{ker}(\Phi)$. Consequently, we have $\mathrm{ker}(\Phi) = \mathrm{Aut}_H^{H, A} ({E}) \cong \mathrm{im}(\iota)$. This confirms that the sequence is exact at the second term.

To demonstrate that the sequence is exact at the third term, consider a pair $(\beta, \alpha) \in \mathrm{ker}(\mathcal{W})$. Since the non-abelian $2$-cocycles $(
\rho,\psi,\phi_,\sigma,\omega)_{(\beta, \alpha)}$ and $(
\rho,\phi,\psi, \sigma,\omega)$ are equivalent, Theorem \ref{thm-inducibility-2} implies that the pair \((\beta, \alpha)\) is inducible.  This means there exists an automorphism  $\gamma \in \mathrm{Aut}_H({E} )$ such that $\Phi (\gamma) = (\beta, \alpha)$. This shows that $(\beta, \alpha) \in \mathrm{im}(\Phi)$. Conversely, if a pair $(\beta, \alpha) \in \mathrm{im}(\Phi)$, then by definition the pair $(\beta, \alpha)$ is inducible. According to Theorem \ref{thm-inducibility-2}, the non-abelian 2-cocycles \( (\rho,\psi,\phi_,\sigma,\omega)_{(\beta, \alpha)} \) and \((\rho, \phi, \psi, \sigma, \omega)\) are equivalent. Therefore, \(\mathcal{W}(\beta, \alpha) = 0\), which implies that \((\beta, \alpha) \in \mathrm{ker}(\mathcal{W})\). Thus, we have \(\mathrm{ker}(\mathcal{W}) = \mathrm{im}(\Phi)\), confirming the exactness of the sequence at the third term.
\end{proof}

\subsection{The spacial case of abelian extension}
 In this section, we consider abelian extensions of post-Lie algebras as a particular case.

\medskip

Let $0 \rightarrow V  \xrightarrow{i} {E}  \xrightarrow{p} A  \rightarrow 0$
 be an abelian extension of post-Lie algebras. The extension is called abelian if $V$ is abelian, i.e. $\cdot_{V}=0$ and $[\,,\,]_{V}=0$. A section $s:A \longrightarrow E$ of $p$
is a linear map satisfying  $p\circ s=\id_{A}$. An action of ${A}$ on ${V}$ consists of linear maps  $\rho\,: V\times A \to V,$ $\psi: V \times A \to V$ and $\phi: A \times V \to V$. This forms a representation for a post-Lie algebra and is independent of the choice of the section $s$.

The concept of equivalence between two abelian extensions can be defined similarly to that of non-abelian extensions. Notably, equivalent abelian extensions give the same representation of $A$ on $V$. Furthermore, there is a one-to-one correspondence between equivalence classes
of abelian extensions of the post-Lie algebra $A$ by $V$ and the second cohomology  group ${\mathcal H}^2(A,V)$. Define the linear maps $\sigma :  A\times A \rightarrow  H,~\omega :  A\times A \rightarrow H$ by
 $$\sigma(a,b):=[s(a),s(b)]_{{E}}-s[a,b]_{A},~
 \omega(a,b):=s(a)\cdot_{{E}} s(b)-s(a\cdot_{A} b).$$ Then the $(\sigma, \omega)$ is an abelian 2-cocycle of $A$ with values in $V$ if it satisfies conditions (A3)--(A5), (A8)--(A10), where the operation on $V$ is trivial.

 Let $\mathrm{Aut}_V ({E} )$ be the set of all automorphisms $\gamma \in \mathrm{Aut}({E} )$ that satisfies $\gamma|_V \subset V$. For any $\gamma \in \mathrm{Aut}_V ({E} )$, it follows that $\gamma|_V \in \mathrm{Aut}(V )$. We can also define a map $\overline{\gamma} : A \rightarrow A$ by
\begin{align*}
\overline{\gamma} (a) := p \gamma s (a), \text{ for } a \in A.
\end{align*}
Note that $\overline{\gamma}$ is independent of the choice of $s$ and is indeed an automorphism. This construction gives rise to a group homomorphism
\begin{align*}
\Phi : \mathrm{Aut}_V ({E} ) \rightarrow \mathrm{Aut}(V ) \times \mathrm{Aut}(A ), ~~ \Phi (\gamma) := (\gamma|_V, \overline{ \gamma}).
\end{align*}
A pair of automorphisms $(\beta, \alpha) \in \mathrm{Aut}(V ) \times \mathrm{Aut}(A )$ is said to be {\bf inducible} if the pair $(\beta, \alpha)$ lies in the image of $\Phi$.

In the following result, we find a necessary and sufficient condition for the question when a pair of automorphisms $(\beta, \alpha) \in \mathrm{Aut}(V ) \times \mathrm{Aut}(A )$ is inducible.

\begin{theorem}\label{thm-inducibility'}
Let $0 \rightarrow V  \xrightarrow{i} {E}  \xrightarrow{p} A  \rightarrow 0$ be an abelian extension of post-Lie algebras. A pair $(\beta, \alpha) \in \mathrm{Aut}(V ) \times \mathrm{Aut}(A )$ is inducible if and only if there exists a linear map $\lambda \in \mathrm{Hom}(A, V)$ satisfying the following conditions:
\begin{align}
&\beta(\rho_{a}x)=\rho_{\alpha (a)}\beta (x),\\
&\beta(\psi_{a}x)=\psi_{\alpha(a)}\beta(x),\\
&\beta(\phi_{a}x)=\phi_{\alpha(a)}\beta(x),\\
\label{comp1}&\beta(\sigma(a_1,a_2))-\sigma(\alpha(a_1),\alpha(a_2))
=\rho_{\alpha (a_1)}(\lambda (a_2))- \rho_{\alpha (a_2)}( \lambda (a_1))-\lambda([a_1,a_2]_{{A}}),\\
\label{comp2}&\beta(\omega(a_1,a_2))-\omega(\alpha(a_1),\alpha(a_2))
=\psi_{\alpha(a_2)}\lambda(a_1)+\phi_{\alpha(a_1)}\lambda(a_2)
-\lambda(a_1\cdot_{{A}}a_2),
\end{align}
for all $a \in A, x \in V$.
\end{theorem}

 A pair $(\beta,\, \alpha) \in \Aut(V) \times \Aut(A)$ is called compatible if the conditionz (\ref{comp1}) and (\ref{comp2}) hold.
Let $0 \rightarrow V  \xrightarrow{i} {E}  \xrightarrow{p} A  \rightarrow 0$ be an abelian extension of $A $ by $V $. For any $(\beta , \alpha) \in \mathrm{Aut}(V ) \times \mathrm{Aut} (A )$, we define $(\sigma_{(\beta, \alpha)},\omega_{(\beta, \alpha)})$ by
\begin{align*}
   \sigma_{(\beta, \alpha)} (a, b) = \beta \circ \sigma (\alpha^{-1}(a), \alpha^{-1}(b)),\quad
   \omega_{(\beta, \alpha)} (a, b) = \beta \circ \omega(\alpha^{-1}(a), \alpha^{-1}(b)).
\end{align*}
for $a, b \in A$.
Then we obtain that $(\sigma_{(\beta, \alpha)},\omega_{(\beta, \alpha)})$ is an abelian $2$-cocycle if  $(\beta,\, \alpha) $ is compatible. We define a set map $\mathcal{W} : \mathrm{Aut}(V ) \times \mathrm{Aut} (A ) \rightarrow {\mathcal H}^2(A,V)$ by
\begin{align}\label{wells-m1}
\mathcal{W}(\beta, \alpha) = [(\sigma_{(\beta, \alpha)},\omega_{(\beta, \alpha)}) - (\sigma,\omega)].
\end{align}
This map is independent of the choice of section. Finally, we fit the Wells map into a short exact sequence.

\begin{theorem}\label{wells-seq1}
Let $0 \rightarrow V  \xrightarrow{i} {E}  \xrightarrow{p} A  \rightarrow 0$ be an abelian extension of post-Lie algebras. Then there is an exact sequence
\begin{align*}
1 \rightarrow \mathrm{Aut}_V^{V, A} ({E} ) \xrightarrow{\iota} \mathrm{Aut}_V ({E} ) \xrightarrow{\Phi} \mathrm{Aut}(V ) \times \mathrm{Aut}(A ) \xrightarrow{\mathcal{W}} {\mathcal H}^2(A,V).
\end{align*}
\end{theorem}

\section*{Acknowledgements}
This research was supported by the National Natural Science Foundation of China (No. 11961049).
We thank Professor Dominique Manchon for valuable suggestions and reminding us the paper \cite{FP} on semidirect products of skew braces.

%
%

%

%


\vskip7pt
\footnotesize{
	\noindent Lisi Bai\\
  School of Mathematics and Statistics,\\
	Henan Normal University, Xinxiang 453007, P. R. China};\\
E-mail address: \texttt{{bailisi0430@163.com}}

\vskip7pt
\footnotesize{
\noindent Tao Zhang\\
School of Mathematics and Statistics,\\
Henan Normal University, Xinxiang 453007, P. R. China;\\
 E-mail address: \texttt{{zhangtao@htu.edu.cn}}

\end{document}